\documentclass[10pt,a4paper]{article}
\usepackage[latin1]{inputenc}
\usepackage{amsmath, amsfonts, amssymb, amsthm, latexsym, graphicx, wrapfig}
\usepackage[left = .8 in, right = .8 in, top=.7 in]{geometry}
\usepackage{fancyhdr}

\newcommand{\Z}{\mathbb{Z}}

\newcommand{\phelp}{p}
\newcommand{\pinv}{Q}
\newcommand{\pinvhelp}{q}
\newcommand{\pinvequal}{\sigma}
\newcommand{\littleO}{o}
\newcommand{\bigO}[1]{\mathcal{O} \left( #1 \right)}

\newcommand{\uc}[1]{(\emph{#1})}

\makeatletter
\newtheorem*{rep@theorem}{\rep@title}
\newcommand{\newreptheorem}[2]{%
\newenvironment{rep#1}[1]{%
 \def\rep@title{#2 \ref{##1}}%
 \begin{rep@theorem}}%
 {\end{rep@theorem}}}
\makeatother

\newreptheorem{theorem}{Theorem}
\newreptheorem{lemma}{Lemma}

\newtheorem{theorem}{Theorem}
\newtheorem{lemma}[theorem]{Lemma}
\newtheorem{proposition}[theorem]{Proposition}
\newtheorem{corollary}[theorem]{Corollary}

\newcommand{\ignore}[1]{}
\newcommand{\secref}[1]{\textbf{Section #1}}

\title{Integer Subsets with High Volume and Low Perimeter}
\author{Patrick Devlin -- prd41@math.rutgers.edu}

\begin{document}
\maketitle
\ignore{\begin{abstract}
In this paper, we explore a certain variation of the so-called ``isoperimetric problem" in which integer subsets take the role of geometric figures.  BLAH
\end{abstract}}
\section{Introduction}\label{section introduction}
One of the most widely-known classical geometry problems is the so-called \textit{isoperimetric problem}, one equivalent variation of which is:
\begin{quote}
If a figure in the plane has area $A$, what is the smallest possible value for its perimeter?
\end{quote}
In the Euclidean plane, the optimal configuration is a circle, implying that any figure with area $A$ has perimeter at least $2 \sqrt{A\pi}$, and this lower bound may be obtained if and only if the figure is a circle.
\paragraph*{}In 2011, Miller et al.~\cite{Miller} extended the isoperimetric problem in a new direction, in which integer subsets took the role of geometric figures.  For any integer subset $A$, they defined its \textit{volume} as the sum over all its elements, and they defined its \textit{perimeter} as the sum of all elements $x \in A$ such that $\{x-1, x+1\} \not \subset A$.  Thus, the volume can be thought of as the sum of all the elements of $A$, and the perimeter can be thought of as the sum of all the elements on the ``boundary" of $A$ (that is to say, the elements of $A$ whose successor and predecessor are not both in $A$).
\paragraph*{}The main focus of \cite{Miller} was to examine the relationships between a set's perimeter and its volume.  More specifically, the authors wanted to answer the corresponding ``isoperimetric question"\footnote{They focused on this question in particular because it turns out that all of the related extremal questions are trivial.}:
\begin{quote}
If a subset of $\{0, 1, \ldots \}$ has volume $n$, what is the smallest possible value for its perimeter?
\end{quote}
Adopting their notation, we will let $P(n)$ denote this value through the duration of this paper\footnote{This is sequence A186053 in OEIS.}.
\paragraph*{}Because their work is so recent, Miller et al. are the only ones who have published on this variation of the isoperimetric problem or on the function $P(n)$.  Their work was to provide bounds for $P(n)$, by which they were able to determine its asymptotic behavior.  Specifically, their main result was
\begin{theorem}\label{theirResult}
\emph{(Miller et al., 2011)} Let $P(n)$ be as defined.  Then $P(n) \thicksim \sqrt{2}n^{1/2}$.  Moreover, for all $n \geq 1$,
\begin{equation}\label{theirBounds}
\sqrt{2}n^{1/2} - 1/2 < P(n) < \sqrt{2} n^{1/2} + (2 n^{1/4} + 8) \log _2 \log _2 n + 58.
\end{equation}
\end{theorem}
\paragraph*{}Their proof of the lower bound will be reproduced in following sections.  However, their proof of the upper bound was found by a construction argument, which we will not reproduce here since we will analytically derive a tighter bound in \textbf{Theorem \ref{myBounds}}.
\paragraph*{}Beyond the inequalities in \eqref{theirBounds} provided by Miller et al., nothing else has been published on $P(n)$ except for some values for small $n$.  It should be noted that \cite{Miller} provides very good bounds on a related function, in which the sets of interest are allowed to have both negative as well as positive elements.  However, this result was also obtained by a construction argument, and it is not relevant to this paper.
\subsection*{Outline of Results}
In this paper, we focus on improving the few results known on $P(n)$, including deriving multiple exact formulas and developing an understanding of its interesting long-term behavior.  Many of these results are stated in terms of an intimately related function, $Q(n)$, which may be briefly defined as
\[
Q(n) := \min_{A \subseteq \{0, 1, \ldots\}} \Big \{ per(A^{c}) : vol(A) = n \Big \}.
\]
Since it proves to be so closely related to $P(n)$, we also provide results on $Q(n)$ throughout the paper.
\paragraph*{}We begin in \secref{\ref{section preliminary}} with several prelimary lemmas including those used in \cite{Miller}.  Then in \secref{\ref{section first recurrences}}, we define auxilary functions, with which we combinatorially derive several recursive formulas for $P(n)$.  We then introduce the function $Q(n)$ and derive similar recursive formulas for it as well.
\paragraph*{}In \secref{\ref{section second recurrences}}, we relate the functions $P(n)$ and $Q(n)$ by providing yet more recurrence relations for both of them, from which we see that each function completely determines the other.  With this in place, we move on to \secref{\ref{section analysis of recurrences}}, in which we use these recurrences to determine several analytic results for $P(n)$ and $Q(n)$, including upper and lower bounds and derivations of their asymptotic behavior.
\paragraph*{}Our work then culminates in \secref{\ref{section good recurrences}}, in which we state and prove the strongest results of the paper.  By appealing to our analytic bounds on $P(n)$ and $Q(n)$, we show that for all sufficiently large values of $n$, the recurrences of \secref{\ref{section second recurrences}} admit certain drastic simplifications.  By then combining this result with rigorous computer calculations, we arrive at the main theorem of the paper\footnote{More adequate introductions of the functions $f$ and $g$ are given in \secref{\ref{section analysis of recurrences}}.}:
\begin{reptheorem}{bestResult}
Let $P(n)$ and $Q(n)$ be as given.  Then if $n \geq 0$ is not one of the $177$ known counterexamples tabulated in \textbf{Table \ref{table:exceptions}} of the appendix (in particular, for all $n > 149,894$), we have
\begin{eqnarray*}
P(n) &=& f(n) + Q(g(n)) \qquad \qquad \text{and}\\
Q(n) &=& 1 + f(n) + P(g(n)),
\end{eqnarray*}
where the functions $f(n)$ and $g(n)$, given by
\[
f(n) := \left \lfloor \sqrt{2n} + 1/2 \right \rfloor = \left [ \sqrt{2n}\right ], \qquad \qquad \text{and} \qquad \qquad g(n) := \dfrac{f(n) [ f(n)+1]}{2} - n,
\]
are the smallest nonnegative integers satisfying $[1+2+3+ \cdots + f(n)] - g(n) = n$.
\end{reptheorem}
With this, we derive several other satisfying and revealing reccurence relations and quasi-explicit representations for $P(n)$ and $Q(n)$.  We also breifly demonstate and discuss the intricate fractal-like symmetry of the graphs of these functions.  We then conclude in \secref{\ref{section conclusion}} by noting applications in the design of algorithms related to this problem and with some open questions for future research.
\paragraph*{}For an earlier version of this paper with somewhat more detailed proofs and expositions, see \cite{arXiv}.

\section{Definitions and Notation}\label{section definitions}
For the reader's possible convenience, a brief list of definitions used throughout the paper is given here.  In each definition, $A$ is assumed to be a subset of $\{0, 1, 2, \ldots\}$, and $n$ and $k$ are assumed to be nonnegative integers.
\begin{itemize}
\item The \textit{boundary} of $A$, $\partial A$, is $\partial A := \{z \in A : \{ z-1, z+1\} \not \subseteq A\}.$  In words, it is the set of elements of $A$ whose successor or predecessor is not in $A$.
\item The \textit{volume} and \textit{perimeter} of $A$ are defined as
\[
vol(A) := \sum_{z \in A} z, \qquad \text{and} \qquad per(A) := \sum_{z \in \partial A} z,
\]
respectively.  (For convention, the volume and perimeter of the empty set is 0.)
\item $P(n) := \min_{A \subseteq \{0, 1, \ldots\}} \Big \{ per(A) : vol(A) = n \Big \}.$
\item The \textit{complement} of $A$ is $A^{c} := \{0, 1, \ldots \} \setminus A = \{z \in \{0, 1, \ldots \} : z \notin A\}.$
\item $Q(n) := \min_{A \subseteq \{0, 1, \ldots\}} \Big \{ per(A^{c}) : vol(A) = n \Big \}.$
\item The helper functions $p(n;k)$ and $q(n;k)$ are defined as
\[
p(n;k) := \min_{A \subseteq \{0, 1, \ldots , k\}} \Big \{per(A) : vol(A) = n \Big \}, \quad \qquad q(n;k) := \min_{A \subseteq \{0, 1, \ldots , k\}} \Big \{per(A^{c}) : vol(A) = n \Big \}.
\]
\item The special helper function $\pinvequal (n;k)$ is $\pinvequal (n;k) := \min_{A \subseteq \{0, 1, \ldots, k\}} \Big \{per(A^c) : vol(A) = n, \quad \text{and} \quad k \in A \Big \}.$
\item The functions $f(n)$ and $g(n)$ are given by
\[
f(n) = \left [\sqrt{2n} \right], \qquad \text{and} \qquad g(n) = \dfrac{f(n) [f(n) + 1]}{2} - n = \dfrac{\left [\sqrt{2n} \ \right] ^2 + \left [\sqrt{2n}\ \right]}{2} - n,
\]
where $[x]$ denotes the nearest integer function.  In \textbf{Proposition \ref{fReps}}, we show that $f(n)$ and $g(n)$ are also the smallest nonnegative integers satisfying $[1+2+\cdots + f(n)] - g(n) = n$.
\item For all $N$ (and particularly, for $N = 149,894$), we define $\phi (n; N) = \phi (n):=\min_{i \geq 0} \{g^{i} (n) \leq N\}$.
\end{itemize}

\section{Preliminary Results}\label{section preliminary}
The following lemma is used throughout \cite{Miller} and is essential in proving their lower bound on $P(n)$.
\begin{lemma}\label{basicInequality}
\emph{(Miller et al., 2011)} Assume $A$ is a finite nonempty subset of $\{0, 1, \ldots \}$, and let $m$ denote its maximum element.  Then
\[
m \leq per(A) \leq vol(A) \leq \dfrac{m(m+1)}{2}.
\]
\end{lemma}
Using this lemma, the following lower bound is immediately attained.
\begin{proposition}\label{lowerBoundForP}
Assume $A \subseteq \{0, 1, \ldots \}$ is finite.  Then we have
\[
\sqrt{2vol(A)} -1/2 \leq \dfrac{-1 + \sqrt{1+8 vol(A)}}{2} \leq per(A).
\]
Moreover, for any positive integer $n$, this implies
\[
\sqrt{2} n^{1/2} - 1/2 \leq P(n).
\]
\end{proposition}
As stated before, except for the previously mentioned constructive upper bound on $P(n)$, these two results are all that has been published about $P(n)$.  The remainder of the paper is devoted to new results.

\subsection*{Miscellaneous Lemmas}
\begin{lemma}
Let $A \neq \emptyset$ be a finite subset of $\{0, 1, \ldots \}$, and let $m$ denote its maximum element.  Then
\[
m+1 \leq per(A^c)
\]
with equality if and only if $\{1, \ldots, m\} \subseteq A$.
\end{lemma}
\begin{proof}
Let $A$ be as given.  Then $m \in A$, but we know $m+1 \notin A$.  Therefore, $m+1 \in \partial A^c$ implying that $m+1 \leq per(A^c)$.  Now since $m+1 \in \partial A^c$, we know that $m+1 = per(A^c)$ if and only if $\partial A^c$ is equal to either $\{m+1\}$ or $\{0, m+1\}$.  But this happens if and only if $\{1, 2, \ldots , m\} \subseteq A$, as desired.
\end{proof}
\begin{proposition}\label{lowerBoundForQ}
Assume $A \subseteq \{0, 1, \ldots \}$ is finite.  Then we have
\[
\sqrt{2vol(A)} + 1/2 \leq \dfrac{-1 + \sqrt{1+8 vol(A)}}{2} + 1 \leq per(A^c).
\]
Moreover, for any positive integer $n$, this implies
\[
\sqrt{2} n^{1/2} + 1/2 \leq Q(n).
\]
\end{proposition}
\begin{proof}
This follows from the previous lemma in the same way as \textbf{Proposition \ref{lowerBoundForP}}.
\end{proof}

\section{Recurrence Relations using Auxilary Functions} \label{section first recurrences}
We now derive our first set of recurrence relations for $P(n)$ and $Q(n)$.  Although the relations derived in \secref{\ref{section second recurrences}} are actually more revealing, the relations presented here follow naturally, and they motivate the introduction of important auxilary functions.  Moreover, because of their convenient structure, these relations are used extensively in the design of algorithms for computing values, as we breifly discuss in \secref{\ref{section conclusion}}.
\subsection*{First Recurrence for $P(n)$}
As is often the case in analyzing discrete functions, we may obtain an exact recurrence relation for $P(n)$ in terms of a related auxillary function.  In our case, recall that $P(n)$ is the minimum perimeter among all subsets of $\{0, 1, \ldots \}$ having volume $n$.  This suggests defining an auxilary function, $\phelp (n;k)$, as
\[
\phelp (n; k) = \min_{A \subseteq \{0, 1, \ldots , k\}} \Big \{ per(A) : vol(A) = n \Big \}.
\]
Then for all $n\geq 0$, we have
\[
P(n) = \min_{k \in \{0, 1, \ldots \}} \Bigg \{ \min_{A \subseteq \{0, 1, \ldots , k\}} \Big \{ per(A) : vol(A) = n \Big \} \Bigg \} = \min_{k \in \{0, 1, \ldots \}} \Bigg \{ \phelp (n;k) \Bigg \}.
\]
\paragraph*{}From its definition, it is clear that for all fixed $n$, the function $\phelp (n; k)$ is monotonically decreasing with $k$.  Moreover, for all $K \geq n$, we have $\phelp (n; K) = \phelp (n; n)$ since any subset of $\{0, 1, \ldots \}$ having volume $n$ must necessarily be a subset of $\{0, 1, 2, \ldots , n\}$.  Therefore the above equation simplifies to
\begin{equation}\label{PRecHelp}
P(n) = \min_{k \in \{0, 1, \ldots \}} \Bigg \{ \phelp (n;k) \Bigg \} = \lim _{k \to \infty} \phelp (n; k) = \phelp (n; n).
\end{equation}
Thus, we now seek a recurrence for $\phelp (n; k)$, which will provide us with $P(n)$ by calculating $\phelp (n; n)$.
\paragraph*{}For notational convenience, let $S(n; k)$ denote the set of all subsets of $\{0, 1, \ldots , k \}$ having volume $n$.  Then to obtain our desired recurrence for $\phelp (n;k)$, we will consider the following paritition of $S(n;k)$
\[
S(n;k) = \bigcup _{l=0} ^{k+1} \Big \{A \in S(n;k) : \{l, \ldots , k\} \subseteq A \quad \text{and} \quad l-1 \notin A \Big \}.
\]
From this partition, it follows that
\begin{equation}\label{pHelpRec1}
\phelp (n;k) = \min_{l \in \{0, 1, \ldots , k+1\}} \Bigg \{ \min_{A \in S(n;k)} \Big\{ per(A) : \{l, \ldots , k\} \subseteq A \quad \text{and} \quad l-1 \notin A \Big \} \Bigg \}.
\end{equation}
\paragraph*{}Now let $0 \leq l \leq k+1$ be fixed.  Then we have
\begin{eqnarray*}
& & \min_{A \in S(n;k)} \Big\{ per(A) : \{l, \ldots , k\} \subseteq A \quad \text{and} \quad l-1 \notin A \Big \}\\
& & \qquad = \min_{B \subseteq \{0,1, \ldots, l-2\}} \Big\{ per(B \cup \{l, l+1, \ldots, k\}) : vol(B \cup \{l, l+1, \ldots, k\})=n \Big \}\\
& & \qquad = \begin{cases}\displaystyle \min_{B \in \{0,1, \ldots, k-1\}} \Big\{ per(B) : vol(B)=n \Big \} \quad &\text{if $l = k+1$},\\ \displaystyle \min_{B \in \{0,1, \ldots, k-2\}} \Big\{ k + per(B) : vol(B)=n-k \Big \} \quad &\text{if $l = k$},\\ \displaystyle \min_{B \in \{0,1, \ldots, l-2\}} \Big\{k + l + per(B) : vol(B)=n - \left [k(k+1)/2 - l(l-1)/2 \right] \Big \} \quad &\text{if $0 \leq l < k$},\end{cases}\\
& & \qquad = \begin{cases}\phelp(n;k-1) \quad &\text{if $l = k+1$},\\ k+ \phelp (n-k; k-2) \quad &\text{if $l = k$},\\ k + l + \phelp \big (n - [k(k+1)-l(l-1)]/2 ; l-2 \big ) \quad &\text{if $0 \leq l < k$}.\end{cases}
\end{eqnarray*}
Therefore, by substituting into \eqref{pHelpRec1}, we are able to obtain the recurrence
\begin{eqnarray}
\phelp (n;k) &=& \min \Bigg \{ \phelp(n;k-1), k + \phelp(n-k;k-2),\nonumber\\
& & \qquad \qquad k+\min_{l \in \{0, \ldots, k-1\}} \Big\{l + \phelp \big (n - [k(k+1)-l(l-1)]/2 ; l-2 \big ) \Big \} \Bigg \}\label{pHelpRec},
\end{eqnarray}
which is valid for all $n \geq 1$ and for all $k \geq 1$.  Moreover, as boundary conditions, which are clear from its definition, we have that $\phelp (n;k)$ satisfies
\[
p(n;k) = \begin{cases}0 \qquad &\text{if $n=0$,}\\ \infty \qquad &\text{if $n < 0$ or $k \leq 0 < n$.}\end{cases}
\]
Thus, this recurrence for $\phelp (n;k)$ gives the following compact recursive representation for $P(n)$ for all $n \geq 0$:
\begin{equation}
P(n) = \min \Big \{p(n; n-1), n \Big \}.
\end{equation}

\subsection*{Introduction of $\pinv (n)$ and Derivation of First Recurrences}
Because of its intimate connections with the function $P(n)$ that will be explored in subsequent sections, we now introduce the function $\pinv (n)$, which is defined as
\[
\pinv (n) = \min_{A \subseteq \{0, 1, \ldots\}} \Big \{ per(A^c) : vol(A)=n \Big \}.
\]
The difference between this function and the function $P(n)$ is subtle, and based on how similarly the two functions are defined, one would expect their behavior to be very close.  As we will see, this is indeed the case, and the connections between $P(n)$ and $\pinv(n)$ are actually of fundamental importance.  However, it is important for the reader to keep in mind the difference in how these functions are defined.

\paragraph*{}As with the function $P(n)$, we define the auxilary function $\pinvhelp (n;k)$ as
\[
\pinvhelp (n;k) = \min_{A \subseteq \{0, 1, \ldots, k\}} \Big \{ per(A^c): vol(A)=n \Big \},
\]
and just as before, for all $n\geq 0$, we have that
\begin{equation}\label{QRecHelp1}
\pinv (n) = \pinvhelp (n;n).
\end{equation}
\paragraph*{}Because of the difference between how the functions $P(n)$ and $\pinv (n)$ are defined, we now need to define a special auxilary function, $\pinvequal (n;k)$, in order to obtain a compact recurrence for $\pinvhelp (n)$.  This function is defined by
\[
\pinvequal (n;k) = \min_{A \subseteq \{0, 1, \ldots, k\}} \Big \{ per(A^c): vol(A)=n \quad \text{and} \quad k \in A \Big \}.
\]
Note the similarities between $\pinvequal (n;k)$ and $\pinvhelp (n;k)$.  In fact, it is clear that for all $n \geq 1$ and $k \geq 0$, we have
\begin{equation}\label{qRec}
\pinvhelp (n; k) = \min_{l \in \{1, 2, \ldots , k\}} \Big \{ \pinvequal (n;l) \Big \}.
\end{equation}
Using this equation and \eqref{QRecHelp1}, we obtain that for all $n \geq 1$
\begin{equation}\label{QRecHelp2}
\pinv (n) = \min_{l \in \{1, 2, \ldots , n\}} \Big \{ \pinvequal (n;l) \Big \},
\end{equation}
with $\pinv (0) = 0$.
\paragraph*{}Just as was the case for $P(n)$, in order to obtain a useful recurrence relation for $Q(n)$, it now only remains to find a recurrence for $\pinvequal (n;k)$.  As before, we accomplish this by a simple partition yielding
\[
\pinvequal (n;k) = k+1 + \min \Big \{\pinvequal (n-k; k-1) - k,  \pinvequal (n-k; k-2), k-1 + \pinvhelp (n-k; k-3) \Big \},
\]
which we obtain by partitioning the subsets of interest into the three groups (I) sets containing $k-1$, (II) sets containing $k-2$ but not $k-1$, and (III) sets containing neither $k-2$ nor $k-1$.

\paragraph*{}At this point, we need to note that some care must be given to the interpretation of the above equation, which depends on how we define $\pinvequal (0;0)$.  However, if we note and state as a boundary condition that $\pinvequal(n,n) = 2n$ for all $n \geq 1$, then these concerns are effectively removed.
\paragraph*{}We then have a recurrence relation for $\pinvequal$.  As boundary conditions for $\pinvequal (n; k)$, we have
\[
\pinvequal (n;k) = \begin{cases}0 \qquad &\text{if $n=k=0$,}\\ 2n \qquad &\text{if $n=k \geq 1$,}\\ \infty \qquad &\text{if $n < 0$ or if $k \in \{0, 1\}$ and $n > k$,}\\ \infty \qquad &\text{if $0 \leq k > n \geq 0$.}\end{cases}
\]
Then for all $n \geq 2$, and $2 \leq k < n$, we have
\[
\pinvequal (n;k) = k+1 + \min \Big \{k-1 + \pinvhelp (n-k; t-3), \pinvequal (n-k; k-2), \pinvequal (n-k; k-1) - k  \Big \}.
\]
Thus, by using \eqref{QRecHelp2} we have a recurrence for $Q(n)$ as well.

\section{More Direct Recurrence Relations}\label{section second recurrences}
Now by making use of different partitions of the sets of interest, we derive the following recurrence relations, from which we see the first connections between the functions $P(n)$ and $Q(n)$.
\subsection*{Recurrence for $P(n)$ involving $\pinvhelp (n;k)$ and $\pinvequal(n;k)$}
We may calculate $P(n)$ by a ``more direct" recurrence relation, which is found by partitioning all sets of volume $n$ first according to their maximum element, $m$, and then according to the largest integer smaller than $m$ not contained in each set.
\paragraph*{}Let $A$ be a set of volume $n$, let $m$ be its maximum element, and let $l$ be the largest element of $\{-1, 0, \ldots, m\}$ not contained in $A$.  Then $A$ may be written uniquely as $A = \{0, 1, 2, \ldots, m\} \setminus B$ for some set $B \subseteq \{0, 1, \ldots , l\}$, where the volume of $B$ is equal to $(1 + 2 + \cdots + m)-n$ and $l \in B$.  If $l=m-1$, then $per(A) = per(B^c)$.  Else, we have $per(A) = m + per(B^c)$.
\paragraph*{}From this observation, we obtain that for all $n \geq 2$
\begin{equation}\label{PGoodRec}
P(n) = \min_{m \geq 1} \Big \{m + \pinvhelp ([1 + 2 + \cdots + m] - n; m-2), \pinvequal ([1 + 2 + \cdots + m] - n; m-1)  \Big \},
\end{equation}
where $\pinvhelp (n;k)$ and $\pinvequal (n; k)$ are defined as earlier.

\subsection*{Recurrence for $Q(n)$ involving $\phelp (n;k)$}
As before, we also have a simple recurrence that can be used to calculate $\pinv (n)$ ``more directly".  Let $A$ be a set of volume $n$ and maximum element $m$.  Then the set $A$ may be written uniquely in the form $A = \{0, 1, 2, \ldots, m\} \setminus B$ for some set $B \subseteq \{0, 1, \ldots , m-1\}$, where the volume of $B$ is equal to $(1 + 2 + \cdots + m) - n$.  Now we know that for all such sets $A$ and $B$, we have $per(A^c) = per(B) + (m+1)$.
\paragraph*{}This observation leads to the simple and beautiful recurrence that for all $n\geq 2$,
\begin{equation}\label{QGoodRec}
\pinv (n) = 1 + \min_{m \geq 1} \Big \{m + \phelp ([1 + 2 + \cdots + m] - n; m-1) \Big \},
\end{equation}
where $\phelp (n; k)$ is as defined earlier.

\section{Analysis of Recurrences}\label{section analysis of recurrences}
Although equations \eqref{PGoodRec} and \eqref{QGoodRec} appear somewhat intractible (and they offer little or no computational advantage over the first recurrences of \secref{\ref{section first recurrences}}), they turn out to be crucial in understanding the behavior of $P(n)$ (and of $Q(n)$ as well).  In \secref{\ref{section good recurrences}}, we are able to greatly simplify these recurrence, but in order to do so, we must first derive some analytic bounds on $P(n)$ and $Q(n)$.
\subsection*{Relevant Lemmas and Notions}
\begin{lemma}
Let $n$ be a positive integer.  Then there exist unique nonnegative integers $f(n)$ and $g(n)$ satisfying
\[
n = [0 + 1 + \cdots + f(n)] - g(n),
\]
where $0 \leq g(n) < f(n)$.  Moreover, $f(n)$ and $g(n)$ are given by\footnote{We will use these explicit functional representations for $f(n)$ and $g(n)$ so that $f(0) = g(0) = 0$ is well-defined.}
\[
f(n) = \left \lceil \dfrac{-1 + \sqrt{1+8n}}{2} \right \rceil, \qquad \text{and} \qquad g(n) = \dfrac{f(n) [ f(n) + 1]}{2} - n.
\]
\end{lemma}
\paragraph*{}Having defined these functions, we may now restate previous lemmas involving $P(n)$ and $Q(n)$ in these terms.  The most important result we will use combines \textbf{Propositions \ref{lowerBoundForP}} and \textbf{\ref{lowerBoundForQ}} as follows:
\begin{corollary} \label{crudeLowerBounds}
Restating earlier results in new notation, for all $n \geq 1$, we have that
\[
P(n) \geq f(n), \qquad \text{and} \qquad Q(n) \geq f(n) + 1.
\]
\end{corollary}

\paragraph*{}Finally, before moving on, we must present two more results on the functions $f(x)$ and $g(x)$.

\begin{proposition}\label{fReps}
Let $f(n) = \left \lceil \dfrac{-1+\sqrt{1+8n}}{2} \right \rceil$ as before.  Then for all integers $n \geq 0$, we have
\[
f(n) = \left \lceil \dfrac{-1+\sqrt{1+8n}}{2} \right \rceil = \left \lceil \sqrt{2n} -1/2 \right \rceil = \left[ \sqrt{2n} \right],
\]
where $[x]$ is the nearest integer function.
\end{proposition}
\begin{proof}
It suffices to show the first part of the stated equation holds, and the fact that $\sqrt{2n}$ is never a half-integer will complete the proof.  Now by way of contradiction, suppose that the first two representations are not equal.  Then this would imply that there exist integers $p \in \Z$ and $n \in \{0, 1, \ldots \}$ such that
\[
\sqrt{2n} -1/2 \leq p < \dfrac{\sqrt{1+8n} - 1}{2},
\]
which implies $8n \leq (2p +1)^2 < 8n+1.$  But since $n$ and $p$ are integers, this forces $8n = (2p +1)^2$, which taken modulo 2 yields a contradiction.
\end{proof}
\begin{proposition}\label{gBound}
Let $f(n)$ and $g(n)$ be defined as before.  Then for all integers $L \geq 0$ and $n \geq 0$, we have
\[
g^{L} (n) \leq 2 \cdot (n/2) ^{1/2^{L}}.
\]
\end{proposition}
\begin{proof}
The proof is by induction on $L$.  If $L=0$, then the claim is trivially true, which establishes the base case.  Now suppose the claim holds for $L = m$.  Then for all $n \geq 0$, we have
\[
g(n) \leq f(n) - 1 < \sqrt{2n} - 1/2 < \sqrt{2n},
\]
which implies $g^{m+1}(n) = g( g^{m} (n) ) < \sqrt{2 \cdot g^{m} (n)}.$  Then using the induction hypothesis and that the square root function is increasing completes the proof.
\end{proof}

\subsection*{Upper Bounds and Asymptotics for $P(n)$ and $\pinv (n)$}
Using the recurrences of \secref{\ref{section second recurrences}}, we now obtain simple upper bounds on $P(n)$ and $Q(n)$, which taken with the last few lemmas, yield good absolute upper bounds in terms of $n$.
\begin{theorem}\label{goodUpperBounds}
Let $f(n)$ and $g(n)$ be defined as before.  Then for all $n \geq 0$, we have the bounds
\begin{eqnarray*}
P(n) &\leq& f(n) + Q(g(n)), \qquad \text{and}\\
Q(n) &\leq& 1 + f(n) + P(g(n)).
\end{eqnarray*}
\end{theorem}
\begin{proof}
For $n=0$ and $n=1$, the two inequalities hold.  Then for all $n \geq 2$, we may appeal to \eqref{PGoodRec} to obtain
\begin{eqnarray*}
P(n) &=& \min_{m \geq 1} \Big \{m + \pinvhelp ([1 + 2 + \cdots + m] - n; m-2), \pinvequal ([1 + 2 + \cdots + m] - n; m-1)  \Big \}\\
&\leq & f(n) + \min \Big \{\pinvhelp (g(n); f(n)-2), \pinvequal (g(n); f(n)-1) \Big \} = f(n) + \pinvhelp(g(n); f(n)-1) = f(n) + Q(g(n)),
\end{eqnarray*}
and the corresponding inequality for $Q(n)$ is proven analogously.
\end{proof}
\begin{corollary}
For all nonnegative integers $n$ and $L$, we have that
\begin{eqnarray*}
P(n) &\leq& L + P( g^{2L} (n)) + \sum_{i=0} ^{2L-1} f(g^{i}(n)), \qquad \text{and}\\
\pinv (n) &\leq & L + \pinv( g^{2L} (n)) + \sum_{i=0} ^{2L-1} f(g^{i}(n)),
\end{eqnarray*}
where $g^i (n)$ is the $i$-fold composition of $g$ evaluated at $n$, and by convention we take $g^0 (n) = n$.
\end{corollary}
\begin{theorem}\label{myBounds}
Let $P(n)$ and $Q(n)$ be as given.  Then $P(n) \sim Q(n) \sim \sqrt{2} n^{1/2}$.  Moreover, for all $n > 2$,
\begin{eqnarray*}
\sqrt{2} n^{1/2} - 1/2 < &P(n)& \leq \sqrt{2}n^{1/2} + (2^{3/4} \cdot n^{1/4} + 1)[\log_2 ( \log_2 (n/2)) - 1]  + 7, \qquad \text{and}\\
\sqrt{2} n^{1/2} + 1/2 < &Q(n)& \leq \sqrt{2}n^{1/2} + (2^{3/4} \cdot n^{1/4} + 1)[\log_2 ( \log_2 (n/2)) - 1]  + 7.
\end{eqnarray*}
\end{theorem}
\begin{proof}
The lower bounds in the asserted inequalities have already been proven.  To prove the upper bounds, we merely combine the results in the last corollary with the past few bounds on $f(n)$ and $g(n)$.  More specifically, assuming $n > 2$, we know from \textbf{Proposition \ref{gBound}} that if $L \geq (\log_2 ( \log_2 (n/2)) - 1)/2$, then
\[
g^{2L}(n) \leq 2 \cdot (n/2) ^{1/2^{(\log_2 ( \log_2 (n/2)) - 1)}} = \cdots = 8.
\]
By considering values of $P(n)$ and $Q(n)$ for $n \leq 8$, we see that $g^{2L}(n) \leq 8$ implies $P(g^{2L}(n)) \leq 7$ and $Q(g^{2L}(n)) \leq 7$.  Now by the last corollary and the past few lemmas, we have
\begin{eqnarray*}
P(n) & \leq & L + P(g^{2L}(n)) + \sum_{i=0} ^{2L-1} f(g^{i}(n)) \leq L + P(g^{2L}(n)) + \sum_{i=0} ^{2L-1} \sqrt{2 g^{i}(n)} + 1/2\\
& \leq & 2L + P(g^{2L}(n)) + \sum_{i=0} ^{2L-1} \sqrt{4 \cdot (n/2) ^{1/2^{i}}} \leq 2L + P(g^{2L}(n)) + \sqrt{2n} + 2 \sum_{i=1} ^{2L-1} \sqrt{(n/2) ^{1/2^{i}}}\\
& \leq & 2L + P(g^{2L}(n)) + \sqrt{2n} + 4L(n/2) ^{1/4}.
\end{eqnarray*}
Then taking $L = (\log_2 ( \log_2 (n/2)) - 1)/2$ proves the bound.  The inequality for $Q(n)$ is proven analogously.
\end{proof}
Note that these bounds on $P(n)$ are slightly better than those of \cite{Miller} stated in \textbf{Theorem \ref{theirResult}}.  Also note that the upper bound on the summation is very crude.  However, these bounds are sufficient for our purposes.

\section{Obtaining \textit{Good} Recurrences for $P(n)$ and $Q(n)$} \label{section good recurrences}
Although the bounds in \textbf{Theorem \ref{myBounds}} are rather good, they reveal nothing about the actual fluctuations of $P(n)$ and $Q(n)$.  And although we have already obtained multiple recurrence relations for finding exact values, these relations all involve auxilary helper functions, multiple variables, and unweildy minimum functions.  In this section, we combine our analytic bounds and combinatorial results to obtain surprisingly simple and satisfying recurrence relations for $P(n)$ and $Q(n)$ and even quasi-explicit formulae.

\subsection*{New Lower Bounds on $P(n)$ and $Q(n)$}
\begin{lemma}\label{infinityBound}
Let $n$ and $k$ be positive integers with $k < f(n)$.  Then $\phelp(n;k),$ $\pinvhelp(n;k)$, and $\pinvequal(n;k)$ are all infinite.
\end{lemma}
\begin{proof}
This follows from the fact that if $k < f(n)$, there are no subsets of $\{0, 1, \ldots, k\}$ with volume $n$.
\end{proof}

\begin{lemma}\label{helperBound}
Let $n$ and $m$ be positive integers with $m > f(n)$.  Then we have
\begin{eqnarray*}
m + \phelp ([1 + 2 + \cdots + m] - n; m-1) &\geq& f(n) + \sqrt{2 (g(n) + f(n)+1)} + 1/2 \qquad \qquad \text{and}\\
m + \pinvhelp ([1 + 2+\cdots +m] - n; m-2) &\geq& f(n) + \sqrt{2 (g(n) + f(n)+1)} + 3/2.
\end{eqnarray*}
\end{lemma}
\begin{proof}
Consider the following chain of inequalities, which uses the simple lower bound in \textbf{Theorem \ref{myBounds}}
\begin{eqnarray*}
\phelp ([1 + 2+\cdots +m] - n; m-1) &\geq& P ([1 + 2+\cdots +m] - n) \geq \sqrt{2 ([1 + 2+\cdots +m] - n)} - 1/2\\
&\geq& \sqrt{2 (g(n) + [f(n)+1] + [f(n)+2] + \cdots + m)} - 1/2\\
&\geq& \sqrt{2 (g(n) + f(n)+1)} - 1/2.
\end{eqnarray*}
Adding $m \geq f(n) +1$ to both sides proves the first inequality, and the second is proven in the same way.
\end{proof}

\begin{lemma}\label{piBound}
Let $n$ and $m$ be positive integers with $m \geq f(n)$.  Then we have
\[
\pinvequal ([1 + 2 + \cdots + m] - n; m-1) \geq 2 f(n)-2.
\]
\end{lemma}
\begin{proof}
We may assume $f(n) \geq 2$, or the claim is trivially true.  Let $A \subseteq \{0, 1, \ldots, m-1\}$ be such that $vol(A)=[1+2+\cdots + m] -n$ and $m-1 \in A$.  By way of contradiction, suppose that $per(A^c) < 2f(n) -2$.
\paragraph*{}If $m \geq 2f(n) - 2$, then since $m-1 \in \partial A$, this would imply that $per(A^c) \geq m \geq 2f(n) -2$.  Therefore, we may assume that $m \leq 2f(n) - 3$.  Now since $m \geq f(n)$, the volume of $A$ may be written as
\[
vol(A)  [1 + 2 + \cdots + m] - n = g(n) + [(f(n) + 1) + (f(n) + 2) + \cdots + m] < f(n) + [f(n) + 1] + \cdots + m,
\]
and because $m \leq 2f(n) -3 = [f(n) - 2] + [f(n) - 1]$, we also have
\[
vol(A) < [f(n)] + [f(n) + 1] + \cdots + [m - 1] + [f(n)-2] + [f(n) -1] = \sum_{i=f(n)-2}^{m-1} i.
\]
\paragraph*{}From this, we know that there is at least one element of $\{f(n)-2, f(n)-1, \ldots , m-2\}$ that is not contained in $A$, because otherwise the volume of $A$ would be too large.  Let $l \in A^c$ be the largest integer satisfying $f(n)-2 \leq l \leq m-2$.  Then since $m-1 \in A$, we know that $l \in \partial A^c$, which implies
\[
per(A^c) \geq l + m \geq f(n) -2 + m \geq f(n) - 2 + f(n) = 2f(n) -2.
\]
But this contradicts the assumption that $per(A^c) < 2f(n) -2$, thus completing the proof.
\end{proof}
\paragraph*{}With these lemmas, we are now able to prove the following lower bounds.
\begin{theorem}\label{goodLowerBounds}
Let $P(n)$ and $Q(n)$ be as given.  Then for all $n \geq 2$, we have
\begin{eqnarray*}
P(n) &\geq & f(n) + \min \Big \{\pinv (g(n)), \sqrt{2 (g(n) + f(n)+1)} + 3/2, f(n)-2 \Big \} \qquad \qquad \text{and}\\
Q(n) &\geq & 1 + f(n) + \min \Big \{P (g(n)), \sqrt{2 (g(n) + f(n)+1)} + 1/2 \Big \}.
\end{eqnarray*}
\end{theorem}
\begin{proof}
Starting with \eqref{PGoodRec} and applying \textbf{Lemmas \ref{infinityBound}, \ref{helperBound},} and \textbf{\ref{piBound}}, we obtain
\begin{eqnarray*}
P(n) &=& \min_{m > f(n)} \Big \{f(n) + \pinvhelp(g(n); f(n)-2), m + \pinvhelp ([1 + 2+\cdots +m] - n; m-2),\\
& & \qquad \qquad \pinvequal(g(n); f(n)-1), \pinvequal ([1 + 2+\cdots +m] - n; m-1)  \Big \}\\
&\geq& f(n) + \min \Big \{\pinv (g(n)), \sqrt{2 (g(n) + f(n)+1)} + 3/2, f(n)-2 \Big \}.
\end{eqnarray*}
The second inequality is proven analogously by starting with \eqref{QGoodRec}.
\end{proof}
\ignore{
Similarly, we also may obtain the following lower bound on $Q(n)$.
\begin{proposition}\label{goodLowerBoundForQ}
Let $n \geq 2$.  Then $Q(n)$ satisfies
\[
Q(n) \geq 1 + f(n) + \min \Big \{P (g(n)), \sqrt{2 (g(n) + f(n)+1)} + 1/2 \Big \}.
\]
\end{proposition}
\begin{proof}Just as before, starting from \eqref{QGoodRec}, for all $n \geq 2$ we have
\begin{eqnarray*}
\pinv (n) &=& 1 + \min_{m \geq 1} \Big \{m + \phelp ([1 + 2 + \cdots + m] - n; m-1) \Big \}\\
&=& 1 + \min_{m \geq f(n)} \Big \{m + \phelp ([1 + 2 + \cdots + m] - n; m-1) \Big \}\\
&\geq & 1 + \min_{m > f(n)} \Big \{f(n) + \phelp (g(n); f(n)-1), m + \phelp ([1 + 2 + \cdots + m] - n; m-1) \Big \}\\
&\geq & 1 + f(n) + \min_{m > f(n)} \Big \{P(g(n)), \sqrt{2 (g(n) + f(n)+1)} + 1/2 \Big \},
\end{eqnarray*}
as desired.
\end{proof}
}

\subsection*{Squeezing an Equation from Inequalities (Eventually)}
At this point, we have simple upper bounds on $P(n)$ and $Q(n)$ provided by \textbf{Theorem \ref{goodUpperBounds}} and nearly simple lower bounds from \textbf{Theorem \ref{goodLowerBounds}}, which are complicated by the ``min" operators.  Suppose we could show that \textit{eventually} $P(g(n))$ and $Q(g(n))$ happen to be the smallest terms in each minimum.  Then our lower bounds would simplify drastically and our lower and upper bounds would squeeze together, yielding a simple pair of mutually recursive equations that would hold for all sufficiently large $n$.
\paragraph*{}As it turns out, we can in fact prove this claim, which is the content of the following proposition:
\begin{proposition}\label{eventuallyHappens}
Let $P(n)$ and $Q(n)$ be as given.  Then there exists an $N \in \mathbb{Z}$ such that for all $n \geq N$
\begin{eqnarray*}
P(g(n)) &=& \min \Big \{P (g(n)), \sqrt{2 (g(n) + f(n)+1)} + 1/2 \Big \} \qquad \qquad \text{and}\\
Q(g(n)) &=& \min \Big \{\pinv (g(n)), \sqrt{2 (g(n) + f(n)+1)} + 3/2, f(n)-2 \Big \}.
\end{eqnarray*}
Moreover, these claims hold if we take $N$ to be $2,500,000$.
\end{proposition}
\begin{proof}
We will first prove there is such an $N \in \Z$.  Then we will discuss why we may take $N$ to be $2,500,000$.
\paragraph*{}We need to show that eventually $P(g(n)) \leq \sqrt{2 (g(n) + f(n)+1)} + 1/2$.  From \textbf{Theorem \ref{myBounds}}, we know
\[
P(r) \leq \sqrt{2r} + \littleO(\sqrt{r}).
\]
Therefore, there exists a constant $G$ such that for all $r \geq G$, we have
\[
P(r) \leq \sqrt{2r} + \littleO(\sqrt{r}) \leq \sqrt{4r}.
\]
From this, it follows that for all $n$, if $g(n) \geq G$, then we have
\[
P(g(n)) \leq \sqrt{4g(n)} \leq \sqrt{2(g(n) + f(n) + 1)} + 1/2.
\]
\paragraph*{}Let $M$ be the maximum value taken by $P(k)$ for $0 \leq k \leq G$, and let $n \geq M^2 (M^2+1)/2$ be arbitrary.  Now if $g(n) \geq G$, then we know the claim holds.  Therefore, we can assume $g(n) < G$.  But if this is the case, then we know $P(g(n)) \leq M$, which implies
\[
P(g(n)) \leq M \leq \sqrt{f(n)} \leq \sqrt{2(g(n) + f(n)+1)}+1/2.
\]
\paragraph*{}Therefore, for all $n \geq M^2 (M^2+1)/2 =:N_P$, the first equation holds.  In the same way, we may find a constant $N_Q$ after which the second inequality holds.  Thus, taking $N := \max \{N_P, N_Q\}$ proves the existence of such an integer $N$.
\paragraph*{}Now proving that we may in fact take $N$ to be $2,500,000$, follows from somewhat lengthy but routine refinements of the previous argument.  In the above notation, the main idea is to first obtain any analytic upper bound on $G$.  This upper bound on $G$ is then refined by using computer calculated data to compare $P(r)$ with $\sqrt{4r}$ to make $G$ as small as possible.  Using this technique for both $N_P$ and $N_Q$ then proves the claim.
\end{proof}
With this proposition, we are able to prove our main result.
\begin{theorem}\label{bestResult}
Let $P(n)$ and $Q(n)$ be as given.  Then if $n \geq 0$ is not one of the $177$ known counterexamples tabulated in \textbf{Table \ref{table:exceptions}} of the appendix (in particular, for all $n > 149,894$), we have
\begin{eqnarray*}
P(n) &=& f(n) + Q(g(n)) \qquad \qquad \text{and}\\
Q(n) &=& 1 + f(n) + P(g(n)),
\end{eqnarray*}
where as before, the functions $f(n)$ and $g(n)$, given by
\[
f(n) := \left \lfloor \sqrt{2n} + 1/2 \right \rfloor = \left [ \sqrt{2n}\right ], \qquad \qquad \text{and} \qquad \qquad g(n) := \dfrac{f(n) [ f(n)+1]}{2} - n,
\]
are also the smallest nonnegative integers satisfying $[1+2+3+ \cdots + f(n)] - g(n) = n$.
\end{theorem}
\begin{proof}
If $n \geq 2,500,000$, then the result follows by using the previous proposition to simplify the lower bounds of \textbf{Theorem \ref{goodLowerBounds}} and comparing these to the upper bounds in \textbf{Theorem \ref{goodUpperBounds}}.
\paragraph*{}On the other hand, if $0 \leq n < 2,500,000$, then the result holds by performing an exhaustive computer seach for counterexamples\footnote{A brief discussion of the algorithms used for this search is provided in \secref{\ref{section conclusion}}.  Code is available on request.}.  There are only $177$ counterexamples in this range, as tabulated in \textbf{Table \ref{table:exceptions}} of the appendix.  In particular, if $n > 149,894$, then the claim holds since $149,894$ is the largest counterexample.
\end{proof}

\subsection*{Corollaries and Remarks}
There are many interesting implications of \textbf{Theorem \ref{bestResult}}; from this result, many things can be discovered about the behavior of $P(n)$ and $Q(n)$, and the intimate connection between these two functions is made evident.  Although these results can be formulated simply as algebraic statements about the recurrence relations, the corresponding geometric statements about the graphs of these functions is perhaps more enlightening.
\begin{figure}[htb]
\centering
\includegraphics[width=\textwidth]{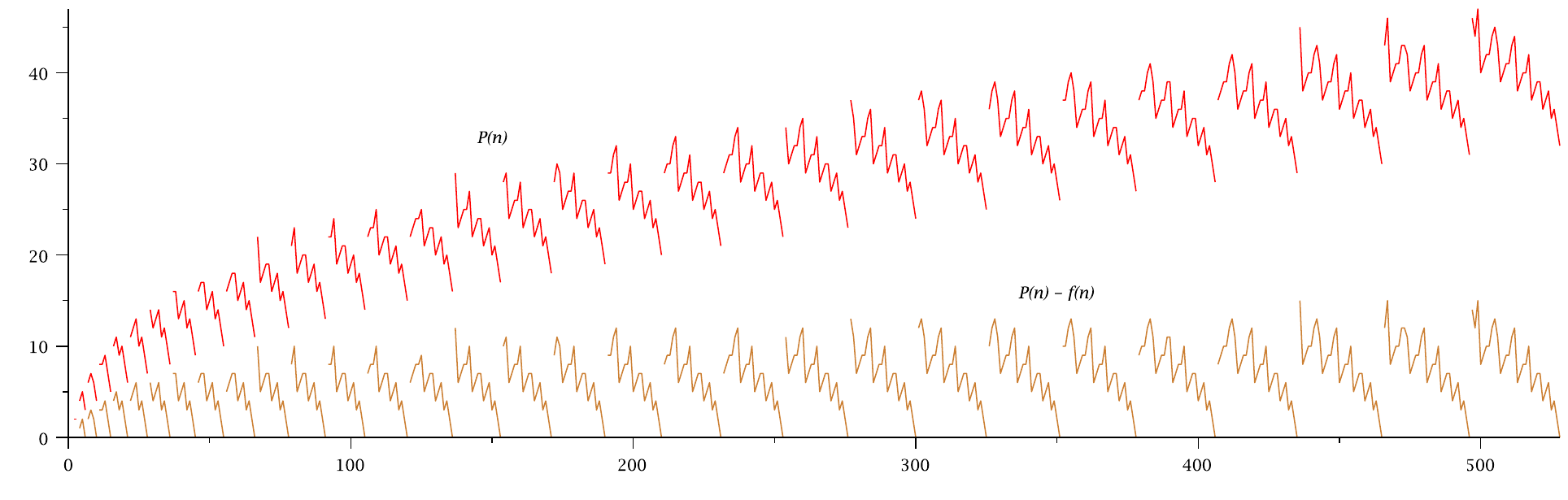}
\caption{Graph of $P(n)$ (\emph{red}) and $P(n) - f(n) = P(n) - [\sqrt{2n}]$ (\emph{brown})}
\label{fig:p_and_p_minus_f}
\end{figure}
\paragraph*{}Examining \textbf{Figures \ref{fig:p_and_p_minus_f}} and \textbf{\ref{fig:q_and_q_minus_f_minus_1}} suggests several apparent patterns of the graphs of these functions.  For example, we see that the graphs $P(n)$ and $Q(n)$ are each ``drifting" upwards by a translation of $f(n)$.  After compensating for this drift, the patterns in the graphs become more apparent.
\begin{figure}[htb]
\centering
\includegraphics[width=\textwidth]{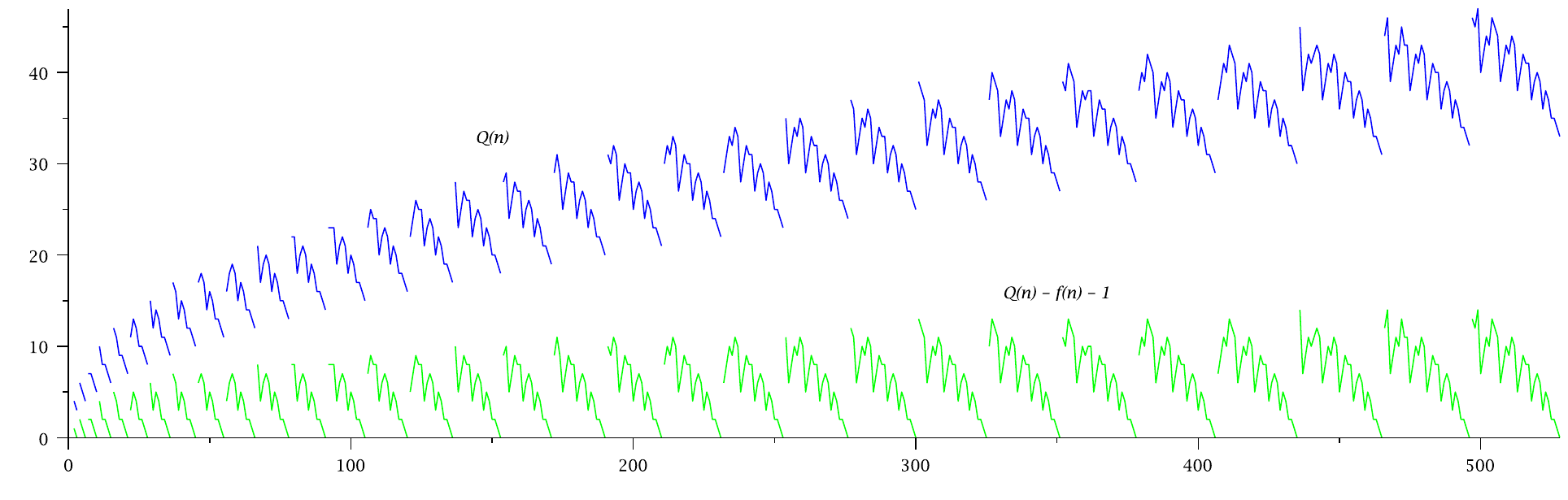}
\caption{Graph of $Q(n)$ (blue) and $Q(n) - f(n) - 1 = Q(n) - [\sqrt{2n}] - 1$ (green)}
\label{fig:q_and_q_minus_f_minus_1}
\end{figure}
\paragraph*{}Now the curves $P(n) - f(n)$ and $Q(n) - f(n) - 1$ (shown in brown and green respectively) appear to be almost ``periodic" in a sense, with zeroes at $0, 1, 3, 6, 10, \ldots$.  This apparent behavior is even more pronounced when the values of these functions are laid out in the following triangular array
\[
\begin{tabular}{cccccc}
\multicolumn{5}{c}{\text{$\{a_{n}\}_{n=0} ^{\infty}$}}\\
 & & & & $a_0$\\
 & & & & $a_1$\\
 & & & $a_2$ & $a_3$\\
 & & $a_4$ & $a_5$ & $a_6$\\
 & $a_7$ & $a_8$ & $a_9$ & $a_{10}$\\
 $a_{11}$ & $a_{12}$ & $a_{13}$ & $a_{14}$ & $a_{15}$\\
 \vdots & \vdots & \vdots & \vdots & \vdots\\
\end{tabular}, \qquad \text{which yields for example} \qquad
\begin{tabular}{cccccc}
\multicolumn{5}{c}{\text{$\{(f(n), g(n))\}_{n=0} ^{\infty}$}}\\
 & & & & $(0,0)$\\
 & & & & $(1,0)$\\
 & & & $(2,1)$ & $(2,0)$\\
 & & $(3,2)$ & $(3,1)$ & $(3,0)$\\
 & $(4,3)$ & $(4,2)$ & $(4,1)$ & $(4,0)$\\
 $(5,4)$ & $(5,3)$ & $(5,2)$ & $(5,1)$ & $(5,0)$\\
 \vdots & \vdots & \vdots & \vdots & \vdots\\
\end{tabular}.
\]
\paragraph*{}Then arranging values in this triangular manner, we have
\[
\begin{tabular}{cccccccccc}
\multicolumn{10}{c}{\text{$\{P(n) - f(n)\}_{n=0} ^{\infty}$}}\\
& & & & & & & & & 0\\
& & & & & & & & & 0\\
& & & & & & & & 0 & 0\\
& & & & & & & 1 & 2 & 0\\
& & & & & & 2 & 3 & 2 & 0\\
& & & & & 3 & 3 & 4 & 2 & 0\\
& & & & 4 & 5 & 3 & 4 & 2 & 0\\
& & & 4 & 5 & 6 & 3 & 4 & 2 & 0\\
& & 6 & 4 & 5 & 6 & 3 & 4 & 2 & 0\\
& 7 & 7 & 4 & 5 & 6 & 3 & 4 & 2 & 0\\
6 & 7 & 7 & 4 & 5 & 6 & 3 & 4 & 2 & 0
\end{tabular}\qquad \qquad \qquad
\begin{tabular}{cccccccccc}
\multicolumn{10}{c}{\text{$\{Q(n) - f(n) - 1\}_{n=0} ^{\infty}$}}\\
& & & & & & & & & -1\\
& & & & & & & & & 0\\
& & & & & & & & 1 & 0\\
& & & & & & & 2 & 1 & 0\\
& & & & & & 2 & 2 & 1 & 0\\
& & & & & 4 & 2 & 2 & 1 & 0\\
& & & & 5 & 4 & 2 & 2 & 1 & 0\\
& & & 3 & 5 & 4 & 2 & 2 & 1 & 0\\
& & 6 & 3 & 5 & 4 & 2 & 2 & 1 & 0\\
& 7 & 6 & 3 & 5 & 4 & 2 & 2 & 1 & 0\\
6 & 7 & 6 & 3 & 5 & 4 & 2 & 2 & 1 & 0
\end{tabular}.
\]
Then it appears that the rows (read from right to left) of $\{P(n) - f(n)\}$ `approach' $0, 2, 4, 3, 6, 5, 4, 7, 7, 6, \ldots$, and the rows of $\{Q(n) - f(n) - 1\}$ `approach' $0, 1, 2, 2, 4, 5, 3, 6, 7, 6, \ldots$.  Moreover, these two sequences seem to be just $\{Q(n)\}$ and $\{P(n)\}$, respectively.  In fact, this follows as our first corollary of \textbf{Theorem \ref{bestResult}}:
\begin{corollary}
Let $\{P(n) - f(n)\}_{n=0} ^{\infty}$ and $\{Q(n) - f(n) - 1\}_{n=0} ^{\infty}$ be arranged in the triangular manner previously discussed.  Then unless $n$ is one of the 177 counterexamples in \textbf{Table \ref{table:exceptions}} of the appendix, reading the rows of $\{P(n) - f(n)\}$ from to right to left exactly agrees with $Q(t)$, and reading the rows of $\{Q(n) - f(n) -1\}$ exactly agrees with $P(t)$.
\end{corollary}
\begin{proof}
This follows immediately from \textbf{Theorem \ref{bestResult}} by how the triangular array was constructed.
\end{proof}
\paragraph*{}Formulating this as a geometric statement is to say that except for 177 particular points, each ``lump" in the graphs of $P(n) - f(n)$ and $Q(n) - f(n) - 1$ is simply a reflection of a partial copy of $Q(n)$ or $P(n)$, respectively.  Thus, the graph of $P(n)$ eventually consists solely of ``shifted" and reflected partial copies of $Q(n)$, and similarly the graph of $Q(n)$ eventually consists solely of ``shifted" and reflected partial copies of $P(n)$.  This mutual similarity of the two functions also induces self-similarity as shown in the following results.
\begin{corollary}
If $g(n) < f(n) - 1$, and if $n$ and $n-f(n)$ are not one of the 177 values in \textbf{Table \ref{table:exceptions}},
\begin{eqnarray*}
P(n) &=& 1 + P(n-f(n)) \qquad \qquad \text{and}\\
Q(n) &=& 1 + Q(n-f(n)).
\end{eqnarray*}
\end{corollary}
\begin{proof}
This follows from \textbf{Theorem \ref{bestResult}} and the fact that if $g(n) \neq f(n)-1$, then $g(n) = g(n-f(n))$.
\end{proof}
This corollary is the statement that with a finite number of exceptions, unless $n$ is one of the values at the far left of a row, then the value for $n$ in the triangle for $\{P(n)\}_{n=0} ^{\infty}$ (or in $\{Q(n)\}_{n=0} ^{\infty}$) is simply one more than the value directly above that entry in the triangle.

\begin{corollary}
If $n$ and $g(n)$ are not one of the 177 values listed in \textbf{Table \ref{table:exceptions}} of the appendix (and in particular, if $g(n) > 149,894$), then we have
\begin{eqnarray*}
P(n) &=& 1 + f(n) + f(g(n)) + P(g^2(n)) \qquad \qquad \text{and}\\
Q(n) &=& 1 + f(n) + f(g(n)) + Q(g^2(n)).
\end{eqnarray*}
\end{corollary}
\begin{proof}
This follows immediately by applying \textbf{Theorem \ref{bestResult}} twice.
\end{proof}

\paragraph*{}This last recurrence is readily `solved' yielding the following quasi-explicit equations.
\begin{proposition}\label{almostExplicit}
For all $n \geq 0$, let $\phi(n;149,894) = \phi(n)$ denote the smallest nonnegative integer satisfying $g^{\phi(n)}(n) \leq 149,894$.  Then for all $n \geq 0$, we have
\begin{eqnarray*}
P(n) &=& \begin{cases} P(g^{\phi(n)}(n)) + \sum_{i=1}^{\phi(n)} f(g^{i-1}(n)) + \phi(n)/2 \qquad &\text{if $\phi(n)$ is even}\\ Q(g^{\phi(n)}(n)) + \sum_{i=1}^{\phi(n)} f(g^{i-1}(n)) +[\phi(n)-1]/2 \qquad &\text{if $\phi(n)$ is odd,}\end{cases} \qquad \qquad \text{and}\\
Q(n) &=& \begin{cases} Q(g^{\phi(n)}(n)) + \sum_{i=1}^{\phi(n)} f(g^{i-1}(n)) + \phi(n)/2 \qquad &\text{if $\phi(n)$ is even}\\ P(g^{\phi(n)}(n)) + \sum_{i=1}^{\phi(n)} f(g^{i-1}(n)) +[\phi(n)+1]/2 \qquad &\text{if $\phi(n)$ is odd.}\end{cases} 
\end{eqnarray*}
\end{proposition}
\begin{proof}
This follows easily from the previous corollary.  Although the function $\phi(n)$ is much too elusive for most honest mathematicians to call these equations truly ``explicit", they ought not be considered recursive.  This is because even though $P$ and $Q$ are referenced on the right-hand side, their arguments are bounded; therefore, by appealing to \textbf{Table \ref{table:exceptions}}, those terms are effectively known.
\end{proof}
This gives rise to the following, perhaps surprising fact:
\begin{corollary}\label{QPDiff}
Let $P(n)$ and $Q(n)$ be as given.  Then for all $n \geq 0$, we have
\[
-1 \leq Q(n) - P(n) \leq 2.
\]
\end{corollary}
\begin{proof}
For all $n \geq 0$, we can appeal to \textbf{Proposition \ref{almostExplicit}} to obtain that
\[
Q(n) - P(n) = \begin{cases} Q(g^{\phi(n)}(n)) - P(g^{\phi(n)}(n)), \qquad &\text{if $\phi(n)$ is even},\\ P(g^{\phi(n)}(n)) - Q(g^{\phi(n)}(n)) + 1, \qquad &\text{if $\phi(n)$ is odd.}\end{cases}
\]
Moreover, for our purposes, we can assume that $\phi(n)$ is one of the 177 counterexamples tabulated in \textbf{Table \ref{table:exceptions}} or else we could continue to appeal to \textbf{Theorem \ref{bestResult}} until this is the case.  But looking at a table of these 177 values, we see that if $k$ is one of those exceptions, then $0 \leq Q(k) - P(k) \leq 2$, which completes the proof.
\end{proof}

\section{Conclusion} \label{section conclusion}
We conclude by discussing applications for computing $P(n)$ and $Q(n)$ and by listing some open questions.
\subsection*{``Sufficiently Large" and Computer Algorithms}
In \textbf{Proposition \ref{eventuallyHappens}}, we state results that hold for all sufficiently large values of $n$ (in particular, for all $n \geq 2,500,000$).  We then use this result to prove \textbf{Theorem \ref{bestResult}}, and we use a computer aided search to completely classify all counterexamples, which brings up a brief discussion of algorithms.
\paragraph*{}The most na\"ive approach to compute $P(n)$ would be simply to list all sets of volume $n$ and find which has the smallest perimeter.  This would require roughly $\bigO{2^n}$ time and $\bigO{n}$ memory, which is much too slow for large $n$, and a different approach is needed.
\paragraph*{}Using the recurrence relations in \secref{\ref{section first recurrences}}, dynamic programming enables us to design algorithms for computing $P(n)$ and $Q(n)$ taking $\bigO{n^2 f(n)} = \bigO{n^{2.5}}$ time and using $\bigO{n^2}$ memory.  We can reduce this memory requirement to roughly $\bigO{n}$ by employing a custom data structure, which benefits from the fact that for fixed $n$, functions such as $p(n;k)$ seem to take very few distinct values.  Using these algorithms, the author was able to check all values of $P(n)$ and $Q(n)$ for $n \leq 3,500,000$, which is more than enough to obtain the results of \textbf{Theorem \ref{bestResult}}.

\paragraph*{}Now that we have proven the recurrences in \textbf{Theorem \ref{bestResult}} and \textbf{Proposition \ref{almostExplicit}}, we may use these to compute $P(n)$ or $Q(n)$ in $\bigO{\Phi(n)} \leq \bigO{\log_{2} \log_{2} (n/2)}$ time using no additional memory.  Moreover, we can compute a list of $P(0), P(1), \ldots , P(n)$ [or $Q(0), Q(1), \ldots , Q(n)$] in $\bigO{n}$ time using the required $\bigO{n}$ memory.
\paragraph*{}Thus, one can now simply use \textbf{Theorem \ref{bestResult}} and the 177 values in \textbf{Table \ref{table:exceptions}} to compute $P(n)$ and $Q(n)$ extremely quickly, and $P(n)$ and $Q(n)$ can be tabulated essentially as far out as desired.  The author is more than willing to provide anyone interested with code and calculated results.

\subsection*{Open Questions}
There are several possible areas of future research.  Because the function $P(n)$ was first introduced so recently, this paper serves as a comprehensive overview of all that is known.
\begin{itemize}
\item[--] Little is known about the behavior of the functions $\phelp(n;k)$, $\pinvhelp(n;k)$, and $\pinvequal(n;k)$.
\item[--] It appears that for any fixed $n \leq 100,000$ the function $\phelp(n;k)$ takes at most two finite values as $k$ varies.  This may be interesting and might be proveable by focusing on \textbf{Proposition \ref{eventuallyHappens}}.
\item[--] Very little or nothing whatsoever is known about $\phi(n;N)$ from \textbf{Proposition \ref{almostExplicit}}.
\item[--] Characterizing sets for which $P(n)$ is obtained may be interesting.  It seems likely that the partitions used and the code developed in this paper would help with that.  Moreover, the result of \textbf{Theorem \ref{bestResult}} seems likely to help with this.
\item[--] Providing more direct (i.e., less analytic) proofs for these results would likely be quite enlightening.
\item[--] There seems to be no pattern or unifying properties for the 177 counterexamples tabulated in \textbf{Table \ref{table:exceptions}}.  Alternate proofs of the main results may shed light on these seemingly sporadic values.
\end{itemize}
\bibliography{mybib}

\begin{thebibliography}{1}

\bibitem{arXiv}
Patrick Devlin.
\newblock Sets with high volume and low perimeter.
\newblock {\em {\tt arXiv:1107.2954v1 [math.CO]}}.

\bibitem{Miller}
Steven~J. Miller, Frank Morgan, Edward Newkirk, Lori Pedersen, and Deividas
  Seferis.
\newblock Isoperimetric sets of integers.
\newblock {\em Mathematics Magazine}, 2011.

\end{thebibliography}

\newpage
\section{Appendix}
The 177 counterexamples to \textbf{Theorem \ref{bestResult}} are tabulated below.  Entries of the form \uc{123} are not actually counterexamples to the theorem, and they are included here only for completeness.
\begin{table}[ht!]
\centering
\begin{tabular}[c]{|c|c|c|}
\hline
\textbf{n}	 & 	\textbf{P(n)}	 & 	\textbf{Q(n)}	\\
\hline
0	 & 	0	 & 	0	\\
2	 & 	2	 & 	\uc{4}	\\
4	 & 	4	 & 	\uc{6}	\\
7	 & 	6	 & 	\uc{7}	\\
8	 & 	7	 & 	\uc{7}	\\\hline
11	 & 	8	 & 	\uc{10}	\\
16	 & 	10	 & 	\uc{12}	\\
17	 & 	11	 & 	\uc{11}	\\
29	 & 	14	 & 	\uc{15}	\\
92	 & 	\uc{22}	 & 	23	\\\hline
125	 & 	25	 & 	\uc{25}	\\
154	 & 	28	 & 	28	\\
155	 & 	29	 & 	\uc{29}	\\
174	 & 	29	 & 	29	\\
361	 & 	\uc{38}	 & 	38	\\\hline
390	 & 	39	 & 	\uc{39}	\\
441	 & 	\uc{42}	 & 	42	\\
473	 & 	43	 & 	43	\\
529	 & 	\uc{46}	 & 	46	\\
564	 & 	47	 & 	47	\\\hline
601	 & 	49	 & 	\uc{50}	\\
637	 & 	49	 & 	49	\\
704	 & 	54	 & 	\uc{55}	\\
742	 & 	53	 & 	53	\\
743	 & 	54	 & 	55	\\\hline
783	 & 	54	 & 	54	\\
837	 & 	\uc{53}	 & 	54	\\
1003	 & 	\uc{58}	 & 	59	\\
1147	 & 	62	 & 	62	\\
1184	 & 	\uc{63}	 & 	64	\\\hline
1340	 & 	67	 & 	67	\\
1341	 & 	68	 & 	\uc{69}	\\
1380	 & 	\uc{68}	 & 	69	\\
1394	 & 	68	 & 	68	\\
1548	 & 	72	 & 	72	\\\hline
1549	 & 	73	 & 	\uc{74}	\\
1606	 & 	73	 & 	73	\\
1665	 & 	74	 & 	74	\\
1771	 & 	77	 & 	77	\\
1772	 & 	78	 & 	\uc{79}	\\\hline
1833	 & 	78	 & 	78	\\
1896	 & 	79	 & 	79	\\
2173	 & 	\uc{82}	 & 	82	\\
2241	 & 	83	 & 	83	\\
2279	 & 	86	 & 	86	\\\hline
\end{tabular}
\begin{tabular}[c]{|c|c|c|}
\hline
\textbf{n}	 & 	\textbf{P(n)}	 & 	\textbf{Q(n)}	\\
\hline
2508	 & 	\uc{88}	 & 	88	\\
2581	 & 	89	 & 	89	\\
2867	 & 	\uc{94}	 & 	94	\\
2945	 & 	95	 & 	95	\\
3250	 & 	\uc{100}	 & 	100	\\\hline
3333	 & 	101	 & 	101	\\
3336	 & 	\uc{103}	 & 	104	\\
3503	 & 	104	 & 	105	\\
3588	 & 	104	 & 	104	\\
3657	 & 	\uc{106}	 & 	106	\\\hline
3745	 & 	107	 & 	107	\\
3748	 & 	\uc{109}	 & 	110	\\
3925	 & 	110	 & 	111	\\
4015	 & 	110	 & 	110	\\
4016	 & 	111	 & 	\uc{112}	\\\hline
4107	 & 	111	 & 	111	\\
4466	 & 	116	 & 	116	\\
4467	 & 	117	 & 	\uc{118}	\\
4563	 & 	117	 & 	117	\\
4564	 & 	118	 & 	\uc{119}	\\\hline
4661	 & 	118	 & 	118	\\
5186	 & 	\uc{123}	 & 	124	\\
5289	 & 	123	 & 	123	\\
5806	 & 	\uc{130}	 & 	131	\\
5915	 & 	130	 & 	130	\\\hline
6026	 & 	131	 & 	131	\\
6461	 & 	\uc{137}	 & 	138	\\
6576	 & 	137	 & 	137	\\
6693	 & 	138	 & 	138	\\
6811	 & 	139	 & 	139	\\\hline
7151	 & 	\uc{144}	 & 	145	\\
7272	 & 	144	 & 	144	\\
7395	 & 	145	 & 	145	\\
7396	 & 	146	 & 	\uc{146}	\\
7436	 & 	\uc{143}	 & 	143	\\\hline
7519	 & 	146	 & 	146	\\
8003	 & 	151	 & 	151	\\
8132	 & 	152	 & 	152	\\
8133	 & 	153	 & 	\uc{153}	\\
8262	 & 	153	 & 	153	\\\hline
8305	 & 	\uc{151}	 & 	151	\\
9222	 & 	\uc{159}	 & 	159	\\
9454	 & 	163	 & 	163	\\
10086	 & 	\uc{163}	 & 	164	\\
10187	 & 	\uc{167}	 & 	167	\\\hline
\end{tabular}
\begin{tabular}[c]{|c|c|c|}
\hline
\textbf{n}	 & 	\textbf{P(n)}	 & 	\textbf{Q(n)}	\\
\hline
10478	 & 	169	 & 	\uc{169}	\\
11200	 & 	\uc{175}	 & 	175	\\
11245	 & 	\uc{172}	 & 	173	\\
11505	 & 	177	 & 	\uc{177}	\\
12261	 & 	\uc{183}	 & 	183	\\\hline
12467	 & 	\uc{181}	 & 	182	\\
12580	 & 	185	 & 	\uc{185}	\\
12583	 & 	\uc{187}	 & 	188	\\
12904	 & 	188	 & 	189	\\
13066	 & 	188	 & 	188	\\\hline
13370	 & 	\uc{191}	 & 	191	\\
13703	 & 	193	 & 	\uc{193}	\\
13752	 & 	\uc{190}	 & 	191	\\
14041	 & 	196	 & 	197	\\
14210	 & 	196	 & 	196	\\\hline
14381	 & 	197	 & 	197	\\
15052	 & 	204	 & 	\uc{205}	\\
15227	 & 	205	 & 	\uc{206}	\\
15402	 & 	204	 & 	204	\\
15403	 & 	205	 & 	206	\\\hline
15580	 & 	205	 & 	205	\\
15759	 & 	206	 & 	206	\\
16511	 & 	\uc{208}	 & 	209	\\
17254	 & 	\uc{214}	 & 	215	\\
17441	 & 	214	 & 	214	\\\hline
17985	 & 	\uc{217}	 & 	218	\\
18955	 & 	223	 & 	223	\\
19152	 & 	\uc{224}	 & 	224	\\
19522	 & 	\uc{226}	 & 	227	\\
20532	 & 	\uc{232}	 & 	232	\\\hline
20533	 & 	233	 & 	\uc{234}	\\
20737	 & 	233	 & 	233	\\
21122	 & 	\uc{235}	 & 	236	\\
21961	 & 	\uc{241}	 & 	242	\\
22172	 & 	241	 & 	241	\\\hline
22173	 & 	242	 & 	\uc{243}	\\
22385	 & 	242	 & 	242	\\
22654	 & 	\uc{241}	 & 	241	\\
22814	 & 	244	 & 	\uc{244}	\\
23656	 & 	\uc{250}	 & 	251	\\\hline
23875	 & 	250	 & 	250	\\
23876	 & 	251	 & 	\uc{252}	\\
24096	 & 	251	 & 	251	\\
24541	 & 	253	 & 	\uc{253}	\\
24598	 & 	\uc{251}	 & 	251	\\\hline
\end{tabular}
\begin{tabular}[r]{|c|c|c|}
\hline
\textbf{n}	 & 	\textbf{P(n)}	 & 	\textbf{Q(n)}	\\
\hline
26855	 & 	262	 & 	262	\\
28726	 & 	\uc{271}	 & 	271	\\
28783	 & 	\uc{268}	 & 	269	\\
28968	 & 	272	 & 	272	\\
30910	 & 	\uc{281}	 & 	281	\\\hline
31161	 & 	282	 & 	282	\\
33174	 & 	\uc{291}	 & 	291	\\
33434	 & 	292	 & 	292	\\
35518	 & 	\uc{301}	 & 	301	\\
35787	 & 	302	 & 	302	\\\hline
36391	 & 	\uc{301}	 & 	302	\\
37147	 & 	307	 & 	307	\\
39125	 & 	\uc{312}	 & 	313	\\
39625	 & 	317	 & 	317	\\
39626	 & 	318	 & 	319	\\\hline
39909	 & 	318	 & 	318	\\
41958	 & 	\uc{323}	 & 	324	\\
44890	 & 	\uc{334}	 & 	335	\\
47921	 & 	\uc{345}	 & 	346	\\
50126	 & 	353	 & 	353	\\\hline
51051	 & 	\uc{356}	 & 	357	\\
53326	 & 	364	 & 	364	\\
53327	 & 	365	 & 	\uc{365}	\\
53655	 & 	365	 & 	365	\\
56625	 & 	375	 & 	375	\\\hline
56626	 & 	376	 & 	\uc{376}	\\
56964	 & 	376	 & 	376	\\
61851	 & 	\uc{389}	 & 	389	\\
65764	 & 	\uc{401}	 & 	401	\\
66129	 & 	402	 & 	\uc{402}	\\\hline
69797	 & 	\uc{413}	 & 	413	\\
70173	 & 	414	 & 	\uc{414}	\\
73950	 & 	\uc{425}	 & 	425	\\
74337	 & 	426	 & 	\uc{426}	\\
78223	 & 	\uc{437}	 & 	437	\\\hline
78621	 & 	438	 & 	\uc{438}	\\
108375	 & 	510	 & 	510	\\
114014	 & 	523	 & 	523	\\
129359	 & 	\uc{554}	 & 	554	\\
136036	 & 	\uc{568}	 & 	568	\\\hline
142881	 & 	\uc{582}	 & 	582	\\
149894	 & 	\uc{596}	 & 	596	\\
\hline
\end{tabular}
\caption{Comprehensive list of exceptions to \textbf{Theorem \ref{bestResult}}.} \label{table:exceptions}
\end{table}
\end{document}